\documentclass[11pt,draft,reqno]{amsart}

\usepackage{amssymb} 
\usepackage{dsfont} 
\usepackage{mathrsfs} 
\usepackage{stmaryrd} 
\usepackage{color}
\usepackage{xcolor}

\usepackage{soul}

\DeclareMathAlphabet{\mathpzc}{OT1}{pzc}{m}{it} 

\newtheorem{Thm}{Theorem}[section]

\newtheorem{Lem}{Lemma}[section]
\newtheorem{Prop}{Proposition}[section]

\newtheorem{Def}{Definition}[section]

\theoremstyle{definition} 
\newtheorem{Ass}{Assumptions}[section]

\theoremstyle{definition} 
\newtheorem*{probl}{Problem}

\theoremstyle{definition}
\newtheorem{Rem}{Remark}[section]

\theoremstyle{definition}

\makeatletter
\@addtoreset{equation}{section}
\makeatother

\newcommand{\eps}{\varepsilon} 

\def\thetabar{\overline{\theta}}

\newcommand{\A}{\mathsf{A}} 

\newcommand\thetaeps{\theta_\eps} 
\newcommand\chieps{\chi_\eps} 
\newcommand\thetaoeps{\theta_{0\eps}}
\newcommand\chioeps{\chi_{0\eps}}

\newcommand\vhat{\widehat{v}}
\newcommand\thetahat{\widehat{\theta}}

\newcommand\fhat{\widehat{f}}
\newcommand\Fhat{\widehat{F}}

\def\thetadiff{{\sl\Theta}}
\def\chidiff{{\mathcal X}}
\def\thetadiffhat{\widehat{{\sl\Theta}}}

\renewcommand\L[1]{L^#1(\Omega)} 
\renewcommand\H[1]{H^#1(\Omega)} 

\newcommand\sign{\text{sign}} 

\newcommand\vuoto{\varnothing} 
\newcommand\parti[1]{\mathscr{P}(#1)} 

\newcommand\enne{\mathbb{N}} 
\newcommand\erre{\mathbb{R}} 

\newcommand\intom{\int_\Omega} 
\newcommand\integr{\int_0^t} 
\newcommand\intdoppio{\int_0^t\!\!\int_\Omega} 

\def\fdual#1#2{{_{V'}\langle #1,#2\rangle_V}}

\newcommand\norma[2]{\Vert #1\Vert_{#2}} 
\newcommand\normaq[2]{\Vert #1\Vert_{#2}^2} 
\def\normaql#1#2{\left\Vert #1\right\Vert_{#2}^2}

\newcommand{\Xtilde}{\widetilde{X}}
\newcommand{\thetatilde}{\widetilde{\theta}}
\newcommand{\chitilde}{\widetilde{\chi}}

\newcommand\function{\longrightarrow} 

\newcommand\en{\mathbb{N}} 

\providecommand{\clint}[1]{\hspace{0.045ex}\left[#1\right]} 

\newcommand\norm[2]{\Vert #1\Vert_{#2}} 

\providecommand{\clsxint}[1]{\hspace{0.1ex}\left[#1\right[\hspace{0.15ex}} 
\providecommand{\cldxint}[1]{\hspace{0.15ex}\left]#1\right]} 
\providecommand{\opint}[1]{\hspace{0.15ex}\left]#1\right[\hspace{0.15ex}} 

\renewcommand{\L}{{\textsl{L}\hspace{0.17ex}}} 
\DeclareMathOperator{\de}{d \! \hspace{0.2ex}} 

\newcommand{\ftilde}{\widetilde{f}} 
\newcommand{\utilde}{\widetilde{u}} 

\newcommand{\Stop}{{\mathsf{S}}} 

\definecolor{blu}{rgb}{0.1,0.1,1}

\definecolor{green}{rgb}{0.0, 0.5, 0.0}

\definecolor{marr}{rgb}{0.63, 0.47, 0.35}

\begin{document}


\title[Stefan problem with phase relaxation]{A Convergence result for a \\ Stefan problem with phase relaxation}

\author{Vincenzo Recupero}

\dedicatory{ \vspace{3ex}
\large 
Dedicated to Pierluigi Colli on the occasion of his $65^{th}$ birthday
\vspace{2ex}}

\address{\textbf{Vincenzo Recupero} \\
        Dipartimento di Scienze Matematiche \\ 
        Politecnico di Torino \\
        Corso Duca degli Abruzzi 24 \\ 
        I-10129 Torino \\ 
        Italy. \newline
        {\rm E-mail address:}
        {\tt vincenzo.recupero@polito.it}}




\begin{abstract}
In this paper we consider the model of phase relaxation introduced in \cite{Vis01}, where an asymptotic analysis is performed toward an integral formulation of the Stefan problem when the relaxation parameter approaches zero. Assuming the natural physical assumption that the initial condition of the phase is constrained, but taking more general boundary conditions, we prove that the solution of this relaxed model converges in a stronger way to the solution of the classical weak Stefan problem.
\end{abstract}


\maketitle


\thispagestyle{empty}


\section{Introduction}

Modelling phase-transition phenomena in a substance attaining two phases (e.g. solid and liquid) in a bounded domain $\Omega$ of the space during the time interval $\clint{0,T}$, one is led to the the energy balance equation
\begin{equation}\label{en bal}
  \frac{\partial}{\partial t}(\theta+\chi) - \Delta\theta = g \qquad 
  \text{in $Q := \Omega \times \clint{0,T}$},
\end{equation}
where for simplicity we have normalized to $1$ all the physical constants. Here the unknowns 
$\theta = \theta(t,x)$ and $\chi = \chi(t,x)$ stand respectively for the temperature and the phase function: 
$(1-\chi)/2$ represents the solid concentration of the solid portion, $(1+\chi)/2$ is the concentration of the liquid portion, and $-1\leq\chi\leq1$, so that it is allowed the existence of mushy regions where the substance is a mixture of the solid and liquid parts (cf., e.g., \cite[p. 99]{Vis96}). In order to describe the evolution of the system, an equation relating $\theta$ and $\chi$ is needed. If $\theta=0$ is the equilibrium temperature at which the two phases can coexist, then we can take the \textsl{equilibrium condition of Stefan type} 
(see, e.g., \cite{Vis96} and the references therein)
\begin{equation}\label{Stef cond}
  \chi \in \sign(\theta) \qquad \text{in $Q$},
\end{equation}
where $\sign$ denotes the multivalued sign graph (i.e. $\sign(r) := -1$ if $r<0,$ $\sign(r) := [-1,1]$ if $r=0,$ 
$\sign(r):=1$ if $r>0$). Problem \eqref{en bal}--\eqref{Stef cond} is usually called \textsl{Stefan problem}. Notice that \eqref{Stef cond} could be written in the equivalent form 
\begin{equation}\label{Stef cond b}
  \sign^{-1}(\chi) \ni \theta \qquad \text{in $Q$},
\end{equation}
$\sign^{-1}$ being the inverse relation of the multivalued $\sign$ graph ($\sign^{-1}(r) := 0$ if 
$r \in \opint{-1,1}$, $\sign^{-1}(-1) := \cldxint{-\infty,0}$, $\sign^{-1}(1) := \clsxint{0,\infty}$).

If dynamic supercooling or superheating effects are to be taken into account, then condition 
\eqref{Stef cond b} is usually replaced by the following \textsl{relaxation dynamics} for the phase variable 
$\chi$ (cf., e.g., \cite{Vis85,Vis96} and their references)
\begin{equation}\label{Vis rel}
  \eps\frac{\partial \chi}{\partial t} + \sign^{-1}(\chi)\ni\theta \qquad \text{in $Q$},
\end{equation}
$\eps$ being a small kinetic positive parameter. Alternatively, the relaxation dynamics can also be modeled by the inclusion
\begin{equation}\label{new rel}
  \eps\frac{\partial \chi}{\partial t} + \chi\in\sign(\theta) \qquad \text{in $Q$},
\end{equation}
which is not equivalent to \eqref{Vis rel}.

The Stefan problem \eqref{en bal}--\eqref{Stef cond} and the Stefan problems with phase relaxation
\eqref{en bal}, \eqref{Vis rel} and \eqref{en bal}, \eqref{new rel} have been extensively studied: see, e.g., \cite{Dam77,DamKenSat94,Vis96} for \eqref{en bal}--\eqref{Stef cond}, 
\cite{Dam77,Vis85, ColGra93,DamKenSat94,Vis96} for \eqref{en bal}, \eqref{Vis rel}, and 
\cite{Vis01, Rec02a} for \eqref{en bal}, \eqref{new rel}. In particular in \cite{Vis85}, uniqueness and existence of \eqref{en bal}, \eqref{Vis rel}, coupled with suitable initial-boundary conditions, are proved in the framework 
of Sobolev spaces, and the solution of the relaxed problem is shown to converge, in a suitable topology,  to the solution of the problem \eqref{en bal}--\eqref{Stef cond} as $\eps\searrow0$. Problem \eqref{en bal}, \eqref{new rel} instead is dealt with in \cite{Rec02a} where existence, uniqueness, and asymptotic analysis of
the Stefan prblem are studied within the same Sobolev setting. Let us also observe that the Stefan problem with phase relaxation can also be studied taking into account a hyperbolic energy balance yielding a finite speed of propagation for the temperature field (see, e.g., \cite{Vis85,ShoWal87,Sh0Wal91,ColRec02,Rec02b,Rec04}).
 
Though models \eqref{en bal}, \eqref{Vis rel} and \eqref{en bal}, \eqref{new rel} are very natural from the analytic point of view, they have some modelling drawbacks. Indeed, as observed in \cite{Vis01}, in \eqref{Vis rel} the rate of the phase $\chi$ does not depend on $\chi$, because the term $\sign^{-1}$ only represents a constraint for the phase, and in \eqref{new rel} the phase depends only on the sign of the temperature $\theta$. One would expect instead that the rate of $\chi$ decays as $\chi$ approaches $1$ and that it also decays as $\theta$ tends to $0$. In order to overcome this 
modelling issue
in \cite{Vis01} the following relaxation dynamics is proposed:
\begin{equation}\label{phi-rel-mod}
  \eps\frac{\partial \chi}{\partial t} = \psi(\theta,\chi) \qquad \textrm{in }Q
\end{equation} 
for a suitable class of regular functions $\psi:\erre^2\rightarrow\erre$ which are increasing in 
$\theta$, decreasing in $\chi$, and such that $\psi(\theta,\chi) = 0$ if and only if 
$\chi \in \sign(\theta)$, or more generally 
\begin{equation}
  \text{$\psi(\theta,\chi) = 0\ $ if and only if $\ \chi \in \alpha(\theta)$},
\end{equation} 
where $\alpha$ is a general
linearly bounded maximal monotone operator in $\erre$, i.e. a continuous increasing graph in $\erre^2$ (cf., e.g, \cite{Bre73}, however in the next section we will provide all the precise definitions needed in the paper). An example, provided in \cite{Vis01}, is
\[
  \psi(\theta,\chi) = p(\theta^+)\frac{1-\chi}{2} + p(-\theta^-)\frac{1-\chi}{2}, 
\]
where $\theta^+=\max\{\theta,0\}$, $\theta^-=\max\{-\theta,0\}$, and $p : \erre \function \clint{-1,1}$ is a function such that $\pi_+=p(\theta^+)$ (respectively $\pi_-=-p-(\theta^-)$) represents the probability of melting a solid particle (respectively ``crystallizing a liquid particle'') in the unit time, with $p(r)r>0$ for every $r \neq 0$.

In \cite{Vis01} the model of phase relaxation \eqref{en bal}, \eqref{phi-rel-mod} is coupled with zero Dirichlet boundary conditions for the temperature, and it is shown that the
solution of the relaxed problem \eqref{en bal}, \eqref{phi-rel-mod} converges in suitable way to a solution of a rather weak formulation of the Stefan problem \eqref{en bal}, \eqref{Stef cond}.
To be more precise it is shown that as $\eps$ approaches zero along a suitable 
subsequence, the solution of the relaxed problem converges to a solution of a time-integral formulation of the Stefan problem \eqref{en bal}, \eqref{Stef cond},
and in general this weaker formulation has not a unique solution. The setting adopted in \cite{Vis01} makes the proofs nontrivial and $L^1$-techniques are needed.

The aim of our present paper is to perform the asymptotic analysis as $\eps$ approaches zero of the model of phase relaxation \eqref{en bal}, \eqref{phi-rel-mod} in the case of a \emph{bounded} graph 
$\alpha$ (which is \emph{physically very natural}, e.g. $\alpha = \sign$), but assuming \emph{more general} boundary conditions for the temperature $\theta$. In this way we are able to use $L^2$-techniques, we obtain a stronger convergence along the entire family $\eps$ (and not along a subsequence), and we find that the limit problem is actually the unique solution of
\eqref{en bal}, \eqref{Stef cond}.

To be more precise concerning our results, we assume that $\alpha$ is a bounded maximal monotone operator 
$\alpha : \erre\rightarrow\mathscr{P}(\erre)$  and we supply the system \eqref{en bal}, 
\eqref{phi-rel-mod} with the rather general initial-boundary conditions described as follows: letting 
$\{\Gamma_0,\Gamma_1\}$ be a partition of the boundary of $\Omega$ into two measurable sets, we take
\begin{alignat}{3}
  & \theta = \theta_D & \qquad & \text{on $\Gamma_0 \times \clint{0,T}$}, \label{b+i cond 1} \\
  & \partial_{\textbf{n}}\theta = -\theta_N & \qquad & \text{on $\Gamma_1 \times \clint{0,T}$}, 
      \label{b+i cond 2} \\
  & \theta(0,\cdot) + \chi(0,\cdot)  = \theta_0 + \chi_0 & \qquad &  \text{in $\Omega$}, 
      \label{b+i cond 3} 
\end{alignat}
where $\theta_D$, $\theta_N$, $\theta_0$, $\chi_0$ are given functions and $\textbf{n}$ is the outward unit vector normal to the boundary of $\Omega$. If we assume that $\theta_D$ is a sufficiently smooth function defined on the cylinder $Q$, that 
$\theta_N : \Gamma_1 \times \clint{0,T} \longrightarrow \erre$ is regular enough, and that there is a  function $u : Q \longrightarrow \erre$ such that $u = \Delta u$ in $Q$, $u = \theta_D$ on 
$\Gamma_0 \times \clint{0,T}$, and  $-\partial_{\textbf{n}} u = \theta_N$ 
on $\Gamma_1 \times \clint{0,T}$ and we set 
$\thetabar_0 := \theta_0 - u(\cdot,0)$. 
 Hence we rewrite all the equations in the new unknown 
$\thetabar := \theta - u$ so that problem
\eqref{en bal}, \eqref{phi-rel-mod}, 
\eqref{b+i cond 1}--\eqref{b+i cond 3} reads, writing again $\theta$ instead of $\thetabar$ for simplicity,
\begin{alignat}{3}
  & \frac{\partial}{\partial t}(\theta+\chi) - \Delta\theta = 
     g - \frac{\partial u}{\partial t} + \Delta u &\qquad& \text{in $Q$}, \label{our system 1} \\
  & \eps\frac{\partial \chi}{\partial t} = \psi(\theta+u,\chi) & \qquad& \text{in $Q$},
      \label{our system 2} \\
  & \theta = 0 & \qquad & \text{on $\Gamma_0 \times \clint{0,T}$,} \label{our system 3a} \\ 
 &  \partial_{\textbf{n}}\theta=0 & \qquad & \text{on $\Gamma_1 \times \clint{0,T}$},
      \label{our system 3b}\\
 & \theta(0,\cdot) + \chi(0,\cdot)  = \theta_0 + \chi_0 & \qquad &  \text{in $\Omega$}.
  \label{our system 4}
\end{alignat}
This formulation has the advantage that the boundary conditions for $\theta$ are homogeneus and the wider generality is incorporated in the $u$-terms in the right-hand side of the balance equation and in the non-linearity $\psi$. As described above, by means of $L^2$-techinques, we will prove 
that the only solution of \eqref{our system 1}--\eqref{our system 4} converges to the solution of \eqref{our system 1}, \eqref{our system 3a}--\eqref{our system 4} coupled with
\[
  \chi \in \alpha(\theta + u) \qquad  \text{in $Q$}
\]
as $\eps \searrow 0$ 
(not only along a suitable subsequence). 

The plan of the paper is the following. In Section \ref{S:results} we list the precise assumptions on the data of the problem and we state our main results. In Section \ref{S:eps-pb} we analyze the relaxed problem \eqref{our system 1}--\eqref{our system 4}. In the final Section \ref{S:limit} we perform the asymptotic analysis as the relaxation parameter $\eps$ goes to zero.


\section{Main results}\label{S:results}

In this section we give the variational formulation of the problems presented in the Introduction and we state our main results. 

The set of integers greater than or equal to $1$ will be denoted by $\enne$. Given  
$p \in \clsxint{1,\infty}$, a measure space $D$, and a real Banach space $B$, then the space of $B$-valued functions on $D$ which are $p$-integrable will be denoted by $\L^p(D; B)$; the vector space of essentially bounded $B$-valued functions on $D$ is denoted by $\L^\infty(D; B)$. These spaces will be endowed with their natural norms defined by 
$\norm{v}{L^p(D;B)} := \left(\int_D\norma{v(x)}{B}^p \de x\right)^{1/p}$
if $p \in \clsxint{1,\infty}$, and by
$\norm{v}{L^\infty(D;B)} := \inf_w \sup_{x \in D}\norma{w(x)}{B}$, where the infimum is taken over all bounded
$\mu$-measurable functions $w$ equal to $v$ $\mu$-almost everywhere, $\mu$ being the measure on $D$. 
If $p=2$ and $B = E$ is a Hilbert space then this norm is induced by the inner product $(v_1,v_2)_{L^2(D;E)} = \int_D (v_1(x), v_2(x))_E \de x$, where $(\cdot,\cdot,)_E$ is the inner product in $E$.
For the theory of integration of vector valued functions we refer, e.g., to \cite[Chapter VI]{Lan93}. 
We will simply set $L^p(D) := L^p(D;\erre)$ for $p \in \clint{1,\infty}$.
If $n \in \enne$, the
$n$-dimensional Lebesgue measure of a set $D \subseteq \erre^n$ will be denoted by $|D|$. 
In the following the locutions ``almost every" and ``almost everywhere" (``a.e.") 
will always refer to the Lebesgue measure. 
If $D \subseteq \erre^n$ is open, we will make use of the Sobolev space
$H^1(D) := \{v \in L^2(D)\ :\ \partial_i v \in L^2(D),\ i=1,\ldots,n\}$ where $\partial_i v$ denotes the partial derivative of $v$ with respect to the $i$-th variable in the sense of distributions (cf., e.g., \cite{Ada75}). 
The symbol $\nabla$ will denote the distributional gradient operator so that $H^1(D)$ is a real Hilbert space if it is endowed with the inner product
\begin{equation}\label{inn-prod H1}
  (v_1,v_2)_{H^1(D)} := \int_D v_1(x) v_2(x) + \int_D \nabla v_1(x) \cdot \nabla v_2(x), 
  \qquad v_1, v_2 \in H^{1}(D),
\end{equation}
which induces the usual norm $\norm{\cdot}{H^1(D)}$.
If $\partial D$ is of Lipschitz class and if $\Gamma_0$ is open in $\partial D$, then the restriction operator $C^\infty(\overline{D}) \function C(\Gamma_0) : v \longmapsto v|_{\Gamma_0}$ can be uniquely continuosly extended to a linear continuous operator $\gamma_{\Gamma_0} : H^1(D) \function L^2(\Gamma_0)$, where $\Gamma_0$ is endowed with the $(n-1)$-dimensional surface (Hausdorff) measure (see, e.g., \cite{LioMag72, Gri85}). The notation $v|_{\Gamma_0} := \gamma_{\Gamma_0}(v)$ is commonly used for a function $v \in H^1(D)$.
If $a, b \in \erre$, $a < b$, we set $L^p(a,b;B) := L^p(\opint{a,b};B)$ for $p \in \clint{1,\infty}$ and, if $B = E$ is a Hilbert space, we define 
$H^1(a,b;E) := \{f \in L^2(a,b;E)\ :\ f' \in L^2(a,b;E)\}$, where $g'$ denotes
the distributional derivative of a function $g : \opint{a,b} \function E$, and the Hilbertian norm defined by 
$
  \norm{f}{H^1(a,b;E)}^2 := \norm{f}{L^2(a,b;E)}^2 + \norm{f'}{L^2(a,b;E)}^2 
$
is used. For the main properties of the Sobolev space $H^1(a,b;E)$ we refer, e.g., to 
\cite[Appendix]{Bre73}. Now we can present our first set of assumptions.

\vspace{1ex}

\begin{Ass}\label{H} 
The following conditions will be used in the paper.
\begin{itemize}
\item[(H1)] 
  $\Omega \subseteq \erre^n$ is a bounded open connected set with  
  Lipschitz boundary $\Gamma := \partial \Omega$. $\Gamma_0$ and $\Gamma_1$ are open subsets of 
  $\Gamma$ such that 
  $\Gamma_0 \cap \Gamma_1 = \vuoto$. If $\overline\Gamma_0$ and $\overline\Gamma_1$ denote the       
  closures of $\Gamma_0$ and $\Gamma_1$ in $\Gamma$, then we assume that 
  $\overline\Gamma_0 \cup \overline\Gamma_1 = \Gamma$ and that
  $\overline\Gamma_0 \cap \overline\Gamma_1$ is of Lipschitz class. We define
  \begin{align}
& H := L^2(\Omega), \\
& V := H^1_{\Gamma_0}(\Omega) :=  \{v \in \H1\ :\ v|_{\Gamma_0} = 0\},
\end{align}
 endowed with their usual inner products, in particular $V$ is endowed with the inner product induced by \eqref {inn-prod H1}.  If $V'$ denotes the topological dual space of $V$, we define the linear continuous operator $\A : V \function V'$ by
\begin{equation}
  \fdual {\A v_1}{v_2} := \intom \nabla v_1 \cdot \nabla v_2,
  \quad v_1,v_2 \in V,
\end{equation}
where $\fdual{\cdot}{\cdot}$ denotes the duality between $V'$ and $V$. The final time of the evolution will be denoted by $T > 0$ and we set $Q := \Omega \times ]0,T[$.
\vspace{1.5ex}
\item[(H2)]
We are given
\begin{equation}\label{phi lip}
   \text{$\psi : \erre^2 \function \erre\ $ Lipschitz continuous with Lipschitz constant $L$.} 
\end{equation}  
\item[(H3)] 
For every $\eps > 0$ we are given
  \begin{equation}
  f_\eps \in L^1(0,T;H) + L^2(0,T;V'), \qquad 
   u_\eps \in L^2(Q), 
  \end{equation}
  where we recall that 
  \[L^1(0,T;H) + L^2(0,T;V') := \{h = h_1+h_2\ :\ h_1 \in L^1(0,T;H),\  h_2 \in L^2(0,T;V')\}\]
  endowed with the norm
  \[
    \norm{h}{L^1(0,T;H) + L^2(0,T;V')} := 
    \inf_{h = h_1 + h_2} \norm{h_1}{L^1(0,T;H)} + \norm{h_2}{L^2(0,T;V')},
  \]
  where the infimum is taken over all the decompositions $h = h_1 + h_2$ with $h_1 \in L^1(0,T;H)$, 
  $h_2 \in L^2(0,T;V')$.
\vspace{1.5ex}  
\item[(H4)] 
For every $\eps > 0$ we are given
\begin{equation}
  \theta_{0\eps} \in L^2(\Omega), \qquad \chi_{0\eps} \in L^2(\Omega). 
  \end{equation} 
\end{itemize}
\end{Ass}

\vspace{1ex}

\begin{Rem}
Let us observe that in assumptions  
(H1), (H2) we do not require that the $(n-1)$-dimensional Hausdorff measure of $\Gamma_0$ is strictly positive.
\end{Rem}

Let us recall that $V \subset H \subset V'$ with dense and compact embeddings, where $V'$ is endowed with its dual norm induced by $V$ and we have identified $H$ with its dual, thus 
\[
  \fdual{e}{v} = (e,v)_H \qquad \forall e \in H, \quad v \in V.
\]

We will also need the following second set of assumptions:

\vspace{1ex}

\begin{Ass}\label{A}
The following conditions will be used in the paper.
\begin{itemize}
\item[(A1)] 
  $\alpha : \erre \rightarrow \parti{\erre}$ is maximal monotone, i.e. if 
  $D(\alpha) := \{r \in \erre\ :\ \alpha(r) \neq \vuoto\}$ then 
  \[
    (s_1 - s_2)(r_1 - r_2) \ge 0 \qquad \forall  r_1, r_2 \in D(\alpha), \quad s_1 \in \alpha(r_1), s_2 \in \alpha(r_2),
  \] 
  and 
  \[
    (\sigma - s)(\rho - r) \ge 0, \quad s \in \alpha(r), \quad r \in D(\alpha) \quad \Longrightarrow \quad 
    \sigma \in \alpha(\rho). 
  \]
  We also assume that $\alpha$ is ``bounded'', i.e. there is a constant $M > 0$ such that 
  \begin{equation}
  |s| \le M \qquad  \forall r \in D(\alpha), \quad \forall s \in \alpha(r).
  \end{equation}
\item[(A2)]
  $\psi : \erre^2 \function \erre$ is the Lipschitz continuous function given in (H2) of Assumptions  
  \ref{H} satisfying \eqref{phi lip} and the following monotonicity condition:
  \begin{equation}\label{phi increasing in theta}
  \big[\psi(\tau_1, \chi) - \psi(\tau_2, \chi)\big](\tau_1-\tau_2) \ge 0 \qquad
  \forall \tau_1, \tau_2, \chi \in \erre,
\end{equation}
\begin{equation}\label{theta decreasing in chi}
    \big[\psi(\tau, \chi_1) - \psi(\tau, \chi_2)\big](\chi_1-\chi_2) \le 0 \qquad
    \forall \chi_1, \chi_2, \tau \in \erre,
\end{equation}
i.e. $\psi(\cdot, \chi)$ is increasing for every $\chi \in \erre$ and $\psi(\tau, \cdot)$ is decreasing for every $\tau \in \erre$.
\vspace{1.5ex}
\item[(A3)] We assume the following ``compatibility'' condition between $\alpha$ and $\psi$:
  \begin{equation}
    \psi(\tau,\chi) = 0 \quad \Longleftrightarrow \quad \chi \in \alpha(\tau) 
    \qquad \forall (\tau,\chi) \in \erre^2. \label{psi(r,s)=0}
  \end{equation}
\item[(A4)] 
We are given
  \begin{equation}
  f\in BV(0,T;L^1(\Omega)), \qquad 
  u \in L^2(Q), 
  \end{equation}
\item[(A5)]
We are given
\begin{equation}
  \theta_{0} \in L^2(\Omega), \qquad \chi_{0} \in L^\infty(\Omega),
  \end{equation} 
  such that  
\begin{equation}\label{ch in sign(th)}
   \chi_0(x) \in \alpha\big(\theta_0(x) + u(0,x)\big) \qquad \text{for a.e. $x \in \Omega$}.
 \end{equation} 
\end{itemize}
\end{Ass}

\vspace{1ex}

\begin{Rem}
Let us observe that \eqref{ch in sign(th)} is equivalent to condition (3.3) in \cite{Vis01}.
\end{Rem}

Let us recall that under condition (H1) of Assumptions  
\ref{H} and conditions (A1), (A4), (A5) of Assumptions  
\ref{A}, it is well-know that the Stefan problem admits a unique solution, i.e. there exists a unique pair 
$(\theta,\chi) : Q \function \erre^2$ such that
\begin{alignat}{3}
  &\theta\in L^2(0,T;V) \cap H^1(0,T;H), \label{wStef th} \\
  & \chi\in L^\infty(Q), \label{wStef ch} \\
  & \theta + \chi \in H^1(0,T;V') \label{wStef th+ch} \\
  & (\theta + \chi)'(t) + \A\theta(t) = f(t)  & \qquad &     
      \text{in $V'$, for a.e. $t \in \opint{0,T}$,} \label{wStef pb 1} \\ 
  & \chi(t,x) \in \alpha\big(\theta(t,x) + u(t,x)\big) & \qquad & \text{for a.e. $(t,x)\in Q$},  
       \label{wStef pb 2} \\
  &(\theta+\chi)(0)=\theta_0+\chi_0 & \qquad& \text{in $V'$}. \label{wStef pb 3}
\end{alignat}
For a proof we refer, for instance, to \cite{ColGra93, DamKenSat94,Vis96}.

Now we state the weak formulation of the model of phase relaxation 
\eqref{our system 1}--\eqref{our system 4}.

\begin{probl}[\textbf{P$_{\eps}$}] Assume that $\eps > 0$ and that Assumptions  
\ref{H} are satisfied.
Find a pair of functions $(\thetaeps,\chieps) : Q \function \erre^2$ satisfying the following conditions:
\begin{alignat}{3}
  & \thetaeps \in L^2(0,T;V) \cap H^1(0,T;V'),  \label{pb-eps theta} \\
  & \chieps \in H^1(0,T;H), \label{pb-eps chi} \\
  & \thetaeps'(t) + \chieps'(t) + \A \thetaeps(t) = f_\eps(t) & \qquad & 
     \textrm{in $V'$,\ for a.e. $t \in \opint{0,T}$}, \label{pb-eps eq.energ.} \\
  & \eps \chieps'(t,x) = \psi\big(\thetaeps(t,x) + u_\eps(t,x),\chieps(t,x)\big) & \qquad & 
     \text{for a.e. $(t,x) \in Q$}, 
      \label{pb-eps eq.fase} \\
  &  \thetaeps(0) = \thetaoeps, & \quad & \textrm{a.e. in}\ \Omega, \label{pb-eps c.i.theta} \\
  &  \chieps(0) = \chioeps , & \quad & \textrm{a.e. in}\ \Omega. \label{pb-eps c.i.chi}
\end{alignat}
\end{probl}

\vspace{1ex}

Let us now introduce a general notation which will hold throughout the paper.

\begin{Def}\label{D:cappello}
For a real Banach space $B$, and for a function $v \in L^1(0,T;B)$ we define 
${\widehat v} : [0,T] \longrightarrow B$ by setting
\begin{equation}\label{cappello}
  {\widehat v}(t) := \integr v(s) \de s, \qquad  t \in [0,T].
\end{equation}
\end{Def}

\vspace{1ex}

We also state the following {\sl Baiocchi-Duvaut-Fr\'emond formulation} of the classical Stefan problem (cf. \cite{Bai71,Duv73,Fre74}). 

\begin{probl}[\textbf{P}]
Find a pair of functions $(\theta,\chi) : Q \function \erre^2$ satisfying the following conditions:
\begin{alignat}{3}
  &\thetahat\in L^\infty(0,T;V)\cap H^1(0,T;H), \label{P pb 1} \\
  & \chi\in L^\infty(Q), \\
  & \theta(t) + \chi(t) + A\thetahat(t) = \fhat(t) + \theta_0 + \chi_0 & \qquad & 
      \text{in $V'$, for a.e. $t \in \opint{0,T}$,}
   \label{P pb 2} \\ 
  & \chi(t,x) \in \alpha\big(\theta(t,x) + u(t,x)\big) & \qquad & \text{for a.e. $(t,x)\in Q$}. \label{P pb 3}
\end{alignat}
A pair $(\theta,\chi)$ satisfying \eqref{P pb 1}--\eqref{P pb 3} is also called a 
\textsl{solution of the Stefan problem in the sense of Baiocchi-Duvaut-Fr\'emond}.
\end{probl}

\vspace{1ex}

Now we state the main results of this paper.

\begin{Thm}\label{E!Peps}
Assume that $\eps > 0$ and that Assumptions  
\ref{H} hold. Then Problem 
\emph{(\textbf{P}$_{\eps}$)} admits a unique solution. Moreover it is well-posed in the sense specified by Proposition 
\ref{cont-dep Peps} below.
\end{Thm}

\begin{Thm}\label{convergence thm}
If Assumptions \ref{H} and \ref{A} are satisfied, then there exists a unique solution 
$(\theta,\chi)$ of Problem {\bf (P)}, and $(\thetahat,\theta,\chi)$ is the weak-star limit in 
$L^\infty(0,T;V)\times L^2(0,T;H)\times L^\infty(0,T;H)$ of the sequence 
$((\thetahat_\eps,\thetaeps,\chieps))_\eps$ as $\eps\searrow0$, where
$(\thetaeps,\chieps)$ is the solution of {\bf (P$_{\eps}$)} and it is assumed that 
$\chi_{0\eps} \in L^\infty(\Omega)$ for every $\eps > 0$ and
\begin{alignat}{3}
  & f_\eps \to f  & \qquad & \text{in $L^1(0,T;L^1(\Omega))$}, \label{convhp1}\\ 
  & u_\eps \to u  & \qquad & \text{in $L^2(Q)$}, \\
  & \theta_{0\eps}  \to \theta_0 & \qquad & 
  \text{in $H$}, \\
  & \chi_{0\eps}  \to \chi_0 & \qquad & 
  \text{in $L^\infty(\Omega)$} \label{cheps0->ch0}
\end{alignat}
as $\eps \searrow 0$. Moreover $(\theta,\chi)$ is also the unique solution of the Stefan problem
\eqref{wStef th}--\eqref{wStef pb 3}.
\end{Thm}

\begin{Rem}
Let us remark that in Theorem \ref{convergence thm} we have that the whole sequence
$(\thetaeps,\chieps)$ converges to $(\theta,\chi)$. 
\end{Rem}

\begin{Rem} Since the usual weak formulation of the Stefan problem is
stronger than the Baiocchi-Duvaut-Fr\'emond one, from the uniqueness 
property stated in Theorem 2.2 we deduce that the solution 
$(\theta,\chi)$ of {\bf (P)} belongs to 
$[L^2(0,T;V)\cap L^\infty(0,T;H)]\times L^\infty(Q)$ and satisfies 
\eqref{wStef th}--\eqref{wStef pb 3}.
\end{Rem}


\section{The problem with phase relaxation}\label{S:eps-pb}

Let us start by proving a continuous dependence result for Problem (\textbf{P$_{\eps}$}).

\begin{Prop}\label{cont-dep Peps}
Under Assumptions  
\ref{H} there exists a constant $C_\eps$, depending on $T$ and on $\eps$, such that if 
\begin{equation}
 f_{\eps i} \in L^1(0,T;H) + L^2(0,T;V'), 
 \quad u_{\eps i} \in L^2(Q), 
 \quad \theta_{0\eps i} \in H, \quad \chi_{0\eps i} \in H, \qquad i=1,2,
\end{equation}
and if the pair $({\thetaeps}_i,{\chieps}_i)$ satisfies \eqref{pb-eps theta}--\eqref{pb-eps c.i.chi} with
$\thetaeps$, $\chieps$, $f_\eps$, $u_\eps$, 
$\thetaoeps$, and $\chioeps$ replaced respectively by ${\thetaeps}_i$, 
${\chieps}_i$, $f_{\eps i}$, $u_{\eps i}$, 
$\theta_{0\eps i}$, and $\chi_{0\eps i}$, $i = 1,2$, then
\begin{align}
 & \normaq{{\thetaeps}_1(t) - {\thetaeps}_2(t)}{H} + \normaq{{\chieps}_1(t) - {\chieps}_2(t)}{H} 
 \le C_\eps\left(
    \normaq{{\thetaoeps}_{1} - {\thetaoeps}_{2}}{H} + \normaq{{\chioeps}_{1} - {\chioeps}_{2}}{H}
    \right) 
    \notag \\
 & \phantom{\le\ } +
 C_\eps\left(\normaq{f_{\eps 1} - f_{\eps 2}}{L^1(0,T;H) + L^2(0,T;V')}
 +  \normaq{u_{\eps 1} - u_{\eps 2}}{L^2(Q)}\right)\label{contdep}
\end{align}  
for every $t \in \clint{0,T}$. Let us remark that $C_\eps$ does not depend on $f_{\eps i}$, 
$\theta_{0\eps i}$, 
$\chi_{0\eps i}$, $({\thetaeps}_i,{\chieps}_i)$, $i = 1,2$. In particular Problem
\emph{(\textbf{P$_{\eps}$})} has at most one solution.
\end{Prop}

\begin{proof}
Let us set 
$\thetatilde_\eps := {\thetaeps}_1 - {\thetaeps}_2$, $\chitilde_\eps := {\chieps}_1 - {\chieps}_2$, 
$\ftilde_\eps := f_{\eps 1} - f_{\eps 2}$, $\utilde_\eps := u_{\eps 1} - u_{\eps 2}$,
$\thetatilde_{0\eps} := {\thetaoeps}_{1} - {\thetaoeps}_{2}$, and 
$\chitilde_{0\eps} := {\chioeps}_{1} - {\chioeps}_{2}$.
Let $f_{\eps k} = f_{\eps kH} + f_{\eps kV}$ be an arbitrary decomposition of $f_{\eps k}$ with
$f_{\eps kH} \in L^1(0,T;H)$ and 
$f_{\eps kV} \in L^2(0,T;V')$ for $k=1,2$, and set $\ftilde_{\eps H} := f_{\eps 1H} - f_{\eps 2H}$, 
$\ftilde_{\eps V} := f_{\eps 1V} - f_{\eps 2V}$.

Moreover for simplicity let us omit the subscript $\eps$ throughout the reminder of this proof. Let us fix $t \in \clint{0,T}$ and let us start by testing the difference of the energy balance equations for $\theta_1$ and $\theta_2$ by $\thetatilde$ and integrate over $\clint{0,t}$, i.e. we consider the difference of the equations \eqref{pb-eps eq.energ.} with $\theta$ and $\chi$ replaced respectively by $\theta_i$ and $\chi_i$, $i=1,2$, we apply it to $\thetatilde$ and we integrate over $\clint{0,t}$ with $t \in \clint{0,T}$. Using  \eqref {pb-eps c.i.theta} we infer that
\begin{align}
  & \frac{1}{2}\normaq{\thetatilde(t)}{H} + 
      \integr\intom\chitilde'(s,x) \thetatilde(s,x)\de x \de s + 
      \integr\intom|\nabla\thetatilde(s,x)|^2 \de x \de s \notag \\
  & = 
   \frac{1}{2}\normaq{\thetatilde_0}{H} + 
   \integr\intom \ftilde_{H}(s,x)\thetatilde(s,x) \de x\de s +
   \integr\fdual{\ftilde_{V}(s)}{\thetatilde(s)} \de s, 
   \label{eps-unique_est0}
\end{align}
therefore using the elementary 
Young inequality 
\begin{align}
  & \frac{1}{2}\normaq{\thetatilde(t)}{H} + 
      \integr\intom\chitilde'(s,x)\thetatilde(s,x) \de x  \de s + 
      \frac{1}{2}\integr\intom|\nabla\thetatilde(s,x)|^2 \de x \de s \notag \\
  & \le \frac{1}{2}\normaq{\thetatilde_0}{H} + 
   \frac{1}{2}\integr\normaq{\ftilde_V(s)}{V'}\, \de s + 
   \integr\norma{\ftilde_H(s)}{H}\norma{\thetatilde(s)}{H} \de s  
  + \frac{1}{2}\integr\normaq{\thetatilde(s)}{H} \de s.   \label{eps-unique_est1}
\end{align}
Exploiting equation \eqref{pb-eps eq.fase} for the phase relaxation and the Lipschitz continuity \eqref{phi lip} of $\psi$, we find $C_1 > 0$ depending on $\eps$, but independent of $\theta_i$,
 $\chi_i$, $u_i$, $\theta_{0 i}$, $\chi_{0 i}$,  $f_i$, such that (omitting for simplicity the integration variables $s$ and $x$ in some lines)
\begin{align}
  & \integr\intom\chitilde'(s,x)\thetatilde(s,x) \de x \de s \notag \\
  & = \frac{1}{\eps}
     \integr\intom 
     \big[\psi\big(\theta_1 + u_1, \chi_1\big)-\psi\big(\theta_2 + u_2,\chi_2\big) \big]
           \thetatilde \de x \de s  
     \notag \\
  & \ge - \integr\intom \frac{L}{\eps} \left(|\thetatilde + \utilde| + |\chitilde|\right)|\thetatilde| \de x\de s 
  \notag\\
  & \ge - \integr\intom \frac{L}{\eps} \left(|\thetatilde|^2 + |\utilde||\thetatilde| + |\chitilde||\thetatilde|\right) \de x\de s \notag\\
  & \ge - C_1\integr\intom (|\thetatilde|^2 + |\utilde|^2+ |\chitilde|^2) \de x\de s. 
 \label{eps-unique_est2}
\end{align}
Now let us multiply the equation \eqref{pb-eps eq.fase} for the phase relaxation by $\chitilde$, and integrate it over $\Omega \times \clint{0,t}$. 
Thanks to \eqref{pb-eps c.i.chi} and to the Lipschitz continuity \eqref{phi lip} of $\psi$, we deduce that
\begin{align}
  & \frac{1}{2}\normaq{\chitilde(t)}{H} = \frac{1}{2}\normaq{\chitilde_0}{H} \notag \\
  & +
  \frac{1}{\eps}\integr\intom
       \big[\psi(\theta_1(s,x) + u_1(s,x),\chi_1(s,x))-\psi(\theta_2(s,x) + u_2(s,x),\chi_2(s,x))\big] \chitilde(s,x)
       \de x \de s \notag \\. 
 & \le \frac{1}{2}\normaq{\chitilde_0}{H}
 + \frac{L}{\eps}\integr\intom
 \left(|\thetatilde + \utilde|+ |\chitilde|\right)|\chitilde| \de x \de s \notag\\
 & \le \frac{1}{2}\normaq{\chitilde_0}{H}
 + \frac{L}{\eps}\integr\intom
 \left(|\thetatilde||\chitilde| + |\utilde||\chitilde|+ |\chitilde|^2\right) \de x \de s.
   \label{eps-unique_est3}
\end{align}
Summing \eqref{eps-unique_est1} and \eqref{eps-unique_est3}, taking into account of 
\eqref{eps-unique_est2}, and using the elementary Young inequality, we obtain that there exists a constant $C_2$ depending on $\eps$, but independent of $\theta_i$, $\chi_i$, $u_i$, $\theta_{0 i}$, $\chi_{0 i}$,  $f_i$, such that
\begin{align}
  & \normaq{\thetatilde(t)}{H} + \integr\intom|\nabla\thetatilde(s,x)|^2 \de x \de s +
      \normaq{\chitilde(t)}{H} \notag \\
  &  \le C_2\left(\normaq{\thetatilde_0}{H} + \normaq{\chitilde_0}{H}  + 
       \integr\normaq{\ftilde_V(s)}{V'} \de s + 
       \integr\norma{\ftilde_H(s)}{H}\norma{\thetatilde(s)}{H} \de s  
       + \integr\normaq{\utilde(s)}{H} \de s\right) \notag \\
  & \phantom{\le \ }  + C_2\left( \integr\normaq{\thetatilde(s)}{H} \de s + 
  \integr\normaq{\chitilde(s)}{H} \de s\right). \notag
\end{align}
Thus an application of
a generalized version of the Gronwall Lemma (cf. \cite[Theorem 2.1]{Bai67}), 
 yields \eqref{contdep}.
\end{proof}

Now we can conclude the proof of Theorem \ref{E!Peps}.

\begin{proof}[Proof of Theorem \ref{E!Peps}]
For simplicity let us omit the subscript $\eps$.
Let us define $\Sigma := \{h \in H^1(0,T;H)\ :\ h(0) = \chi_0\}$ so that $\Sigma$ is a complete metric space when it is endowed with the metric induced by the norm of $H^1(0,T;H)$. 
Fix $X \in \Sigma$.
Then, thanks to a standard result for parabolic equations
(cf., e.g., \cite[Theorem 3.2]{Bai71}), 
there exists a unique 
$\theta_X \in L^2(0,T;V) \cap H^1(0,T;V')$ such that 
\begin{alignat}{3}
  & \theta_X' + \A\theta_X = f - X' & \qquad & \text{in $V'$, for a.e. $t \in \opint{0,T}$}, \label{fix pt theta} \\
  & \theta_X(0) = \theta_0, & \qquad & \text{a.e. in $\Omega$}. \label{fix pt theta0} 
\end{alignat}
Now define $\chi : Q \function \erre$ by
\begin{equation}\label{contraction1}
  \chi(t,x) := \chi_0(x) + \frac{1}{\eps} \integr \psi(\theta_X(s,x) + u(s,x),X(s,x)) \de s, 
  \qquad t \in \clint{0,T}, \ x \in \Omega.
\end{equation}
Using \eqref{contraction1}, the Lipschitz continuity of $\psi$, and the fact that $\theta_X$, $X$ and $u$ belong to $L^2(0,T;H)$, it is immediately seen that $\chi \in \Sigma$. Hence we can define the operator
$\Stop :  \Sigma \function \Sigma$ associating to 
$X$ the unique $\chi$ satisfying 
\eqref{fix pt theta}--\eqref{contraction1}. We have that $\chi$ is a fixed point of $\Stop$ if and only if 
$(\theta_\chi,\chi)$ is a solution to Problem (\textbf{P}$_{\eps}$). We are going to apply the shrinking fixed point theorem. For 
$i = 1,2$, fix $X_i \in \Sigma$ and let $\theta_i \in L^2(0,T;V) \cap H^1(0,T;V')$ be the unique function such that \eqref{fix pt theta}--\eqref{fix pt theta0} hold with $\theta_X$ and $X$ replaced by $\theta_i$ and $X_i$. Set $\chi_i := \Stop(X_i)$ and define $\Xtilde := X_1 - X_2$, $\thetatilde := \theta_1 - \theta_2$, 
$\chitilde := \chi_1 - \chi_2$. Let $t \in \clint{0,T}$ be fixed.
Let us integrate in time the difference of equations \eqref{fix pt theta} for 
$i = 1,2$ and test it by $\thetatilde$. Integrating the result over $\opint{0,t}$, 
applying the Young inequality, and observing that $\Xtilde(0) = 0$,
we infer that 
\begin{align}
   \frac{1}{2}\integr\intom|\thetatilde(s,x)|^2\de x \de s + 
  \frac{1}{2}\intom\left|\integr\nabla\thetatilde(s,x)\de s\right|^2 \de x \le
  \frac{1}{2}
  \integr\intom|\Xtilde(s,x)|^2 \de x \de s.  
\end{align}
Therefore, using \eqref{contraction1} and the Lipschitz continuity \eqref{phi lip} of $\psi$, we get
\begin{align}
  & \int_0^t\intom|\chitilde'(s,x)|^2 \de x \de s \notag \\
  & = \frac{1}{\eps^2}\int_0^t\intom
         \left|\psi\big(\theta_1(s,x) + u(s,x), X_1(s,x)\big) - \psi\big(\theta_2(s,x) + u(s,x),X_2 (s,x)\big)\right|^2\de x \de s \notag \\
  & \le \frac{L^2}{\eps^2}\int_0^t\intom\big(|\thetatilde(s,x)|^2 + |\Xtilde(s,x)|^2\big) \de x \de s 
   \le  \frac{2L^2}{\eps^2}\int_0^t\intom|\Xtilde(s,x)|^2\de x \de s.   
  \label{contraction2}
\end{align}
On the other hand 
\[
  \int_0^t\intom|\Xtilde(s,x)|^2 \de x \de s 
  \le \int_0^t\intom s\int_0^s|\Xtilde'(r,x)|^2 \de r \de x \de s
  \le \int_0^t\intom t\int_0^s|\Xtilde'(r,x)|^2 \de r \de x \de s, 
\]
hence
\begin{align}
   \int_0^t\intom|\chitilde'(s,x)|^2 \de x \de s 
   \le \frac{2tL^2}{\eps^2}\int_0^t\int_0^s\intom|\Xtilde'(r,x)|^2\de x \de r \de s, 
\end{align}
so that there exists a constant $C$ independent of $X_1$ and $X_2$ such that
\begin{equation}
  \normaq{\chitilde}{H^1(0,t;H)} \le C \integr\normaq{\Xtilde}{H^1(0,s;H)} \de s.
\end{equation}
This entails that 
$\norma{\Stop^n(\Xtilde)}{H^1(0,t;H)} \le ((CT)^n/n!)^{1/2} \norma{\Xtilde}{H^1(0,t;H)}$ for every 
$n \in \en$, therefore the iterated mapping $\Stop^n$ is a strict contraction for $n$ sufficiently large, and
consequently $\Stop$ admits a unique fixed point in $\Sigma$, which leads to the solution we are looking for.
\end{proof}


\section{Asymptotic behavior}\label{S:limit}

Throughout this section we will assume the non restrictive condition that $\eps < 1$.

Let us start by stating the following easy consequence of the assumptions on the function $\psi$, as already observed in \cite[formula (1.12)]{Vis01}.

\begin{Lem}
Under the Assumptions \ref{H} and \ref{A} we have that
\begin{align}
  & \psi(\tau,\chi) > 0 \quad \Longleftrightarrow \quad  \chi < \inf\alpha(\tau), \\
  & \psi(\tau,\chi) < 0 \quad \Longleftrightarrow \quad  \chi > \sup\alpha(\tau),  \label{miss lab}
\end{align}
for every $(\tau,\chi) \in \erre^2$.
\end{Lem}

Now we prove that if the initial datum $\chi_0$ is constrained by $\alpha^{-1}$ (cf. 
\eqref{ch in sign(th)}), then the solution $\chieps$ of \eqref{pb-eps eq.fase} is uniformly bounded on $Q$.

\begin{Lem}\label{|ph|< 1}
Under the Assumptions \ref{H} and \ref{A}, 
if 
we are given  $\chieps \in H^1(0,T;H)$, $\theta_\eps \in L^2(Q)$, and $\eta_\eps > 0$ such that
\begin{equation}
   |\chi_{\eps}(0,x)| \le |\chi_{0}(x)| + \eta_\eps \qquad  \text{for a.e. $x \in \Omega$}
\end{equation}
and
\begin{equation}
   \eps\chieps'(t,x) = \psi\big(\thetaeps(t,x) + u_\eps(t,x),\chieps(t,x)\big)  \qquad  
      \text{for a.e. $(t,x) \in Q$}, \label{ph-rel-eq-|ph|< 1} \\
\end{equation}
then 
\begin{equation}
  |\chieps(t,x)| \le M + \eta_\eps \qquad \text{for a.e. $(t,x) \in Q$},
\end{equation}
where we recall that $M$ is defined in condition (A1) of Assumptions  
\ref{A}, so that 
$M \ge \sup\{\alpha(\tau): \tau \in D(\alpha)\}$.
\end{Lem}

\begin{proof}
Since $\chieps$ and $\chieps'$ belong to $L^2(Q)$,
if $\varphi \in C^\infty(0,T)$ has compact support and if $z \in L^2(\Omega)$, by the Fubini theorem we have that 
\begin{align}
  & \int_\Omega z(x)\int_0^T(\chieps(t,x)\varphi'(t)+\chieps'(t,x)\varphi(t))\de t \de x  \notag \\
  & = \int_0^T\int_\Omega z(x)(\chieps(t,x)\varphi'(t)+\chieps'(t,x)\varphi(t))\de x \de t  \notag \\
  & = \int_0^T(z,\chieps(t))_H \varphi'(t)\de t +
        \int_0^T(z,\chieps'(t))_H \varphi(t) \de t.  \notag
\end{align}   
From the previous chain of equalities, recalling that $\chieps \in H^1(0,T;H)$ and using again the Fubini theorem, we infer that       
\begin{align}
&  \int_\Omega z(x)\int_0^T(\chieps(t,x)\varphi'(t)+\chieps'(t,x)\varphi(t))\de t \de x  \notag \\
& = \int_0^T\Big(z,\chieps(0) + \int_0^t\chieps'(s) \de s\Big)_H \varphi'(t)\de t +
        \int_0^T(z,\chieps'(t))_H \varphi(t) \de t  \notag \\
  & = (z,\chieps(0))_H\int_0^T \varphi'(t)\de t +
         \int_0^T\int_0^t(z,\chieps'(s))_H \varphi'(t)\de s\de t +
        \int_0^T(z,\chieps'(t))_H \varphi(t) \de t  \notag \\ 
  & =  \int_0^T(z,\chieps'(s))_H\int_s^T\varphi'(t) \de t\de s +
        \int_0^T(z,\chieps'(t))_H  \varphi(t)\de t  \notag \\
  & =          -\int_0^T(z,\chieps'(s))_H\varphi(s)\de s +
        \int_0^T(z,\chieps'(t))_H  \varphi(t)\de t  =0,        \notag
\end{align}
whence, by the arbitrariness of $z$, we infer that 
$\int_0^T(\chieps'(t,x)\varphi(t) +\chieps(t,x)\varphi'(t))\de t = 0$ for a.e. $x \in \Omega$, i.e. that 
$\chieps'(\cdot,x)$ is the distributional derivative of $\chieps(\cdot,x)$ for a.e. $x \in \Omega$. Since 
$\chieps'(\cdot,x) \in L^1(0,T)$ for a.e. $x \in \Omega$, we obtain 
that there exists a measurable set $A \subseteq \Omega$ such that $|\Omega \smallsetminus A| = 0$ and
\[
  \chieps(t,x) = \chi_{\eps}(0,x) + \int_0^t \chieps'(s,x) \de s \qquad \forall t \in \clint{0,T}, 
  \qquad \forall x \in A.
\]
It follows that for every $x \in A$ the function $\chieps(\cdot,x)$ is absolutely continuous from 
$\clint{0,T}$ into $\erre$. It is not restrictive to assume that 
$\chi_0(x) \in \alpha(\theta_0(x) + u(0,x))$ for every $x \in A$, so that $|\chi_0(x)| \le M$  for every 
$x \in A$. Therefore $|\chi_\eps(0,x)| \le M + \eta_\eps$ for every $x \in A$.
Let us 
fix $x \in A$ and prove that $|\chieps(t,x)| \le M + \eta_\eps$ for every $t \in \clint{0,T}$. Indeed, if this were not true, there would exist $t_0 \in \opint{0,T}$ such that 
$|\chieps(t_0,x)| > M + \eta_\eps$.  
Let us first assume that $\chieps(t_0,x) > M + \eta_\eps$. Then, by continuity, there exists 
$a_0 \in \clsxint{0,t_0}$ such that $\chieps(a_0,x) = M + \eta_\eps$, and 
$\chieps(t,x) > M + \eta_\eps$ for every $t \in \cldxint{a_0,t_0}$. In particular 
$\chieps(t,x) > \sup\{\alpha(r)\ :\ r \in D(\alpha)\}$, hence 
$\chieps(t,x) > \sup\{\alpha\big(\thetaeps(t,x) + u_\eps(t,x), \chieps(t,x)\big)\}$ 
for every 
$t \in \cldxint{a_0,t_0}$, so that 
$\psi\big(\thetaeps(t,x) + u_\eps(t,x), \chieps(t,x)\big) < 0$ 
for every $t \in \cldxint{a_0,t_0}$ by \eqref{miss lab}. 
It follows that 
$\chieps'(t,x) < 0$ for a.e. $t \in \cldxint{a_0,t_0}$, therefore, as $\chieps(\cdot,x)$ is absolutely continuous, we infer that $\chieps(\cdot,x)$ is decreasing on $\cldxint{a_0,t_0}$, a contradiction. An analogous argument can be used in the case $\chieps(t_0,x) < -M - \eta_\eps$.   
\end{proof}

We need the following auxiliary lemma, where we make use of the notation \eqref{cappello} introduced in Definition \ref{D:cappello}: 
$\widehat{v}(t) = \int_0^t v(s) \de s$, for $t \in \clint{0,T}$, $v \in L^1(0,T;B)$, and a Banach space $B$. 

\begin{Lem}\label{L:lem for est}
Under the hypothesis (H1) in Assumptions \ref{H}, 
if $F \in L^1(0,T;H) + L^2(0,T;V')$, $e_0 \in H$, 
$v \in L^2(0,T;V) \cap L^\infty(0,T;H)$, and $\delta > 0$, then, recalling notation \eqref{cappello}, we have that
\begin{align}
  \integr\fdual{\Fhat(s)}{v(s)} \de s 
  & \le \delta\left(1 + t + \frac{t^2}{2}\right)\normaq{v}{L^2(0,t;H)} + 
  \delta\left(\normaq{\nabla\vhat(t)}{H^n}+\int_0^t\normaq{\nabla\vhat(s)}{H^n}\de s \right) \notag \\
  & \phantom{\le\ } + \frac{1+t}{4\delta} \normaq{F}{L^1(0,T;H) + L^2(0,T;V')} \label{lem for estA}
\end{align}
and
\begin{equation}
\integr\fdual{e_0}{v(s)} \de s \le \delta\normaq{v}{L^2(0,t;H)} + \frac{t}{4\delta}\normaq{e_0}{H} 
 \label{lem for estB}
\end{equation}
for every $t \in \clint{0,T}$.
\end{Lem}

\begin{proof}
Let $F_1 \in L^1(0,T;H)$ and $F_2 \in L^2(0,T;V')$ be arbitrarily taken so that $F = F_1 + F_2$. We have that
\begin{align}
    \integr\fdual{\Fhat_1(s)}{v(s)} \de s 
    & \le \integr \norma{\Fhat_1(s)}{H} \norma{v(s)}{H} \de s \notag \\ 
    & \le \delta \normaq{v}{L^2(0,t;H)} + \frac{1}{4\delta}\normaq{\Fhat_1}{L^2(0,t;H)} \notag \\ 
    & =  \delta \normaq{v}{L^2(0,t;H)} + 
       \frac{1}{4\delta}\integr\normaql{\int_0^s F_1(r) \de r}{H} \de s  \notag \\
    & \le  \delta \normaq{v}{L^2(0,t;H)} + 
      \frac{1}{4\delta}\integr \left( \int_0^s \norma{F_1(r)}{H} \de r \right)^2 \de s  \notag \\
    & \le \delta \normaq{v}{L^2(0,t;H)} + 
      \frac{1}{4\delta}t\normaq{F_1}{L^1(0,T;H)}. \label{lem for est1}     
\end{align}
Let us observe that for any Banach space $B$ we have
\begin{equation}
  \normaq{\vhat(t)}{B} = \normaql{\integr v(s)\de s}B \le \left( \integr\norma{v(s)}B \de s \right)^2 \le 
  t\normaq{u}{L^2(0,t;B)}, \label{lem for est2}
\end{equation}
therefore, integrating by parts and applying Young inequality, we find that
\begin{align}
  & \integr\fdual{\Fhat_2(s)}{v(s)} \de s \notag \\
  & =   \fdual{\Fhat_2(t)}{\vhat(t)} - \integr\fdual{F_2(s)}{\vhat(s)} \de s \notag \\
  & \le \norma{\Fhat_2(t)}{V'} \norma{\vhat(t)}{V} + \integr\norma{F_2(s)}{V'} \norma{\vhat(s)}{V} \de s \notag \\
  & = \norma{\Fhat_2(t)}{V'}\left(\normaq{\vhat(t)}{H} + \normaq{\nabla\vhat(t)}{H^n}\right)^{1/2} +  
         \integr\norma{F_2(s)}{V'} \left(\normaq{\vhat(s)}{H} + \normaq{\nabla\vhat(s)}{H^n}\right)^{1/2} \de s \notag \\ 
  & \le \delta \left(t\normaq{v}{L^2(0,t;H)}+\normaq{\nabla\vhat(t)}{H^n}\right) + 
       \frac{1}{4\delta}\normaq{\Fhat_2(t)}{V'} \notag \\
  & \phantom{\le\ } +
      \delta\integr\left(s\normaq{v}{L^2(0,s;H)}+\normaq{\nabla\vhat(s)}{H^n}\right) \de s +
        \frac{1}{4\delta} \normaq{F_2}{L^2(0,T;V')} \notag \\
  & \le \delta \left(t\normaq{v}{L^2(0,t;H)}+\normaq{\nabla\vhat(t)}{H^n}\right) + 
       \frac{t}{4\delta}\normaq{F_2}{L^2(0,t;V')} \notag \\   
  & \phantom{\le\ } +
      \delta (t^2/2)\normaq{v}{L^2(0,t;H)} + \delta \integr\normaq{\nabla\vhat(s)}{H^n} \de s +
        \frac{1}{4\delta} \normaq{F_2}{L^2(0,T;V')}, \label{lem for est3}           
\end{align}
thus \eqref{lem for estA} follows from \eqref{lem for est1}, \eqref{lem for est3}, and from the 
elementary inequality $a^2 + b^2 \le (|a| + |b|)^2$, holding for $a, b \in \mathbb{R}$. Finally estimate 
\eqref{lem for estB} is a consequence of \eqref{lem for est2} and of formula
\[
  \integr\fdual{e_0}{v(s)}\de s = \fdual{e_0}{\vhat(t)} \le \norm{e_0}{H} \norm{\vhat(t)}{H}.
\]
\end{proof}

We can now deduce the estimate for the temperature $\theta$.

\begin{Lem}\label{L:theta-estimate}
Under the assumptions of Theorem \ref{convergence thm},
there exists a constant $C_1$ independent of $\eps$, but depending on $T$, $\Omega$, $\alpha$, 
$\psi$, $f$, $u$, 
$\theta_0$, $\chi_0$, such that if $(\thetaeps,\chieps)$ is the only solution of Problem \emph{(\textbf{P$_{\eps}$})}, then, recalling notation \eqref{cappello},
\begin{equation}
  \norma{\thetaeps}{L^2(0,T;H)} + \norma{\thetahat_\eps}{L^\infty(0,T;V)} +
  \eps^{1/2}\norma{\thetaeps}{L^2(0,T;V)} \le C_1.
\end{equation}
\end{Lem}

\begin{proof}
We will tacitly use the convergences \eqref{convhp1}--\eqref{cheps0->ch0}.
Let us fix $t \in \clint{0,T}$. First we integrate the energy balance equation \eqref{pb-eps eq.energ.} with respect to time over $\clint{0,s}$ with $s \in \clint{0,t}$, and test it by $\thetaeps(s)$. After a further integration over $\clint{0,t}$, and recalling \eqref{cappello}, we get
\begin{align}
  & \normaq{\thetaeps}{L^2(0,t;H)} 
  + \integr\intom \chieps(s,x)\thetaeps(s,x) \de x \de s
   + \frac{1}{2}\intom |\nabla\thetahat_\eps(t,x)|^2 \de x \notag \\
  & = \integr\fdual{\theta_{\eps 0} + \chi_{\eps 0} + \fhat_\eps(s)}{\thetaeps(s)}\de s,    
  \label{estimate0}
\end{align}
therefore using Lemma \ref{L:lem for est} we infer that there exists a constant $K_1$
depending on $\norm{\theta_0}{H}$, $\norm{\chi_0}{H}$, $\norm{f}{L^1(0,T;H)+L^2(0,T;V')}$, and $T$, but independent of $\eps$, such that 
\begin{align}
  & \frac{1}{2}\normaq{\thetaeps}{L^2(0,t;H)} 
  + \integr\intom \chieps(s,x)\thetaeps(s,x) \de x \de s
  + \frac{1}{4}\intom |\nabla\thetahat_\eps(t,x)|^2 \de x \notag \\
  & \le  K_1 + K_1 \integr \intom |\nabla\thetahat_\eps(s,x)|^2 \de x \de s. 
  \label{estimate1}
\end{align}
Let us recall that $f_\eps = f_{\eps 1} + f_{\eps 2}$ with $f_{\eps 1} \in L^1(0,T;H)$ and 
$f_{\eps 2} \in L^2(0,T;V')$.
We test by $\eps\thetaeps$ the energy balance equation \eqref{pb-eps eq.energ.} and integrate over 
$\clint{0,t}$. Thanks to \eqref {pb-eps c.i.theta},  we infer that
\begin{align}
& \frac{\eps}{2}\normaq{\thetaeps(t)}{H} + 
    \eps\integr\intom\chieps'(s,x)\thetaeps(s,x) \de x \de s 
  + \eps\integr\intom|\nabla\thetaeps(s,x)|^2 \de x \de s \notag \\
  & =    \frac{\eps}{2}\normaq{\thetaoeps}{H} + 
  \eps\integr\left(f_{\eps 1}(s), \thetaeps(s)\right)_H\de s +
  \eps\integr\fdual{f_{\eps 2}(s)}{\thetaeps(s)}\de s 
  \label{estimate2}
\end{align}
therefore, recalling that $\eps < 1$, several applications of Young and H\"older inequlities yield
\begin{align}
& \frac{\eps}{2}\normaq{\thetaeps(t)}{H} + 
    \eps\integr\intom\chieps'(s,x)\thetaeps(s,x) \de x \de s 
  + \frac{\eps}{2}\integr\intom|\nabla\thetaeps(s,x)|^2 \de x \de s \notag \\
  & \le K_2 + K_2\integr\norma{f_{\eps 1}(s)}{H} \eps^{1/2}\norma{\thetaeps(s)}{H} \de s + 
  K_2\integr\eps\normaq{\thetaeps(s)}{H} \de s
  \label{estimate3}
\end{align}
for some $K_2 > 0$ depending on $\theta_0$, on $f$,
but independent of $\eps$.
Thanks to equation 
\eqref{pb-eps eq.fase} for the phase relaxation and to the monotonicity 
\eqref{phi increasing in theta} of $\psi$ in the first variable, we can write (omitting in some lines the integration variable $(s,x)$):
\begin{align}
  & \eps\integr\intom\chieps'(s,x)\thetaeps(s,x) \de x \de s \notag \\
  & = \integr\intom \psi(\thetaeps+u_\eps,\chieps) \thetaeps \de x \de s \notag \\ 
  & = \integr\intom \psi(\thetaeps+u_\eps,\chieps)(\thetaeps+u_\eps) \de x \de s  
       - \integr\intom \psi(\thetaeps+u_\eps,\chieps)u_\eps \de x \de s\notag \\ 
  & = \integr\intom \big[\psi(\thetaeps+u_\eps,\chieps) -  
         \psi(0,\chieps)\big](\thetaeps +u_\eps) \de x \de s \notag \\ 
  & \phantom{=\ } + \integr\intom \psi(0,\chieps)(\thetaeps+u_\eps) \de x \de s -  
     \integr\intom \psi(\thetaeps+u_\eps,\chieps)u_\eps \de x \de s\notag \\ 
  & \ge  \integr\intom \psi(0,\chieps)(\thetaeps+u_\eps) \de x \de s -  
     \integr\intom \psi(\thetaeps+u_\eps,\chieps)u_\eps \de x \de s. \label{estimate4} 
 \end{align}
On the other hand, recalling  the Lipschitz continuity \eqref{phi lip} of $\psi$, we get that (we still omit the integration variable $(s,x)$):
\begin{align}
 & \integr\intom \psi(0,\chieps)(\thetaeps+u_\eps) \de x \de s - 
     \integr\intom \psi(\thetaeps+u_\eps,\chieps)u_\eps \de x \de s\notag \\
  & = \integr\intom \big[\psi(0,\chieps) - \psi(\thetaeps+u_\eps,\chieps)\big]u_\eps \de x \de s 
      + \integr\intom \psi(0,\chieps)\thetaeps \de x \de s \notag \\
  & = \integr\intom \big[\psi(0,\chieps) - \psi(\thetaeps+u_\eps,\chieps)\big]u_\eps \de x \de s 
  \notag \\
  &  \phantom{=\ }  + \integr\intom \big[\psi(0,\chieps)-\psi(0,0)\big]\thetaeps \de x \de s  
 + \integr\intom \psi(0,0)\thetaeps \de x \de s \notag \\ 
  & \ge - \integr\intom L|\thetaeps+u_\eps||u_\eps| \de x \de s 
           - \integr\intom L|\chieps||\thetaeps| \de x \de s
        - \integr\intom |\psi(0,0)||\thetaeps| \de x \de s. \label{estimate5}
\end{align}
Let us observe that thanks to \eqref{cheps0->ch0} and to Lemma \ref{|ph|< 1}, we have that there exists $M_1 >0$ (depending on $M$) such that
\begin{equation}\label{M1}
  \norma{\chi_\eps}{\infty} \le M_1
\end{equation}
for every $\eps < 1$. Therefore, collecting together \eqref{estimate4}--\eqref{estimate5}, and using the elementary Young inequality, we infer that there exists a constant $K_3 > 0$ depending only on $T$, $|\Omega|$, $L$, $|\psi(0,0)|$, $\norma{u}{L^2(0,T;H)}$, and $M$, such that
\begin{equation}
  \eps\integr\intom\chieps'(s,x)\thetaeps(s,x) \de x \de s 
  \ge -\ K_3 - \frac{1}{8}\normaq{\thetaeps}{L^2(0,t;H)}.
  \label{estimate6}
\end{equation}
Using again the boundedness of $\norma{\chi_\eps}{\infty}$ and the elementary Young inequality we also have that
\begin{equation}
  \integr\intom \chieps(t,x) \theta(t,x) \de x \de t \ge 
  -2M_1^2t|\Omega| - 
  \frac{1}{8}\normaq{\thetaeps}{L^2(0,t;H)}.
  \label{estimate7}
\end{equation}
Therefore adding \eqref{estimate6} and \eqref{estimate7}, and taking into account of \eqref{estimate3} and
\eqref{estimate4}, we find a constant $K$ with the same dependencies of $K_1, K_2, K_3$, but independent of $\eps$, such that
\begin{align}
  & \frac{1}{4}\normaq{\thetaeps}{L^2(0,t;H)} + \frac{1}{4}\intom |\nabla\thetahat_\eps(t,x)|^2 \de x + 
  \frac{\eps}{4}\normaq{\thetaeps(t)}{H} + 
  \frac{\eps}{2}\integr\intom|\nabla\thetaeps(s,x)|^2 \de x \de s \notag \\ 
  & \le K + K 
  \left(\integr \intom |\nabla\thetahat_\eps(s,x)|^2 \de x \de s + 
     \integr\norma{f_{\eps 1}(s)}{H} \eps^{1/2}\norma{\thetaeps(s)}{H} \de s + 
  \integr\eps\normaq{\thetaeps(s)}{H} \de s\right), \notag
\end{align}
which, together with a generalized version of the Gronwall Lemma (cf. \cite[Theorem 2.1]{Bai67}), allows us to conclude.
\end{proof}

Now we establish the estimate for the phase $\chi$.

\begin{Lem}\label{L:chi-estimate}
Under the assumptions of Theorem \ref{convergence thm},
there exists a constant $C_2$ independent of $\eps$, but depending on $T$, $\Omega$,  $\alpha$, $\psi$, $f$, $u$, 
$\theta_0$, $\chi_0$, such that if $(\thetaeps,\chieps)$ is the only solution of Problem \emph{(\textbf{P$_{\eps}$})}, then
\begin{equation}
  \norma{\chieps}{L^\infty(Q)} + \eps\norma{\chieps'}{L^2(Q)}  \le C_2.
\end{equation}
\end{Lem}

\begin{proof}
We already know that the sequence $\chi_\eps$ is bounded in $L^\infty(Q)$ by virtue of Lemma \ref{|ph|< 1}.
From equation 
\eqref{pb-eps eq.fase} for the phase relaxation and from the Lipschitz continuity 
\eqref{phi lip} of $\psi$, we get that
\begin{align}
  & \eps^2\integr\intom|\chieps'(s,x)|^2 \de x \de s = 
  \integr\intom|\psi(\thetaeps(s,x) + u_\eps(s,x),\chieps(s,x))|^2 \de x \de s\notag \\ 
  & \le   2\integr\intom|\psi(\thetaeps(s,x) + u_\eps(s,x),\chieps(s,x)) - \psi(0,0)|^2 \de x \de s 
     +  2\integr\intom|\psi(0,0)|^2 \de x \de s\notag \\
  & \le 2L^2\integr\intom(|\thetaeps(s,x) + u_\eps(s,x)|^2+|\chieps(s,x)|^2) \de x \de s + 
         2T|\Omega||\psi(0,0)|^2 \notag \\ 
  & \le 
  4L^2\normaq{\theta}{L^2(0,t;H)} + 4L^2\normaq{u}{L^2(Q)} + 2L^2t|\Omega|M_1^2 + 2T\Omega|\psi(0,0)|^2, \notag 
\end{align}
where $M_1$ is the constant found in \eqref{M1} thanks to \eqref{cheps0->ch0} and to Lemma 
\ref{|ph|< 1}. We conclude by invoking Lemma \ref{L:theta-estimate}.
\end{proof}

We are now ready to prove the main result of this paper.  

\begin{proof}[Proof of Theorem \ref{convergence thm}]
From Lemma \ref{L:theta-estimate} and Lemma \ref{L:chi-estimate} we deduce that there exist two functions 
\begin{equation}
  \theta \in L^2(Q), \qquad\chi \in L^\infty(Q)
\end{equation}
such that, at least for a subsequence which we do not relabel,
\begin{alignat}{3}
  &\thetaeps\ {\rightharpoonup}\ \theta & \qquad & \textrm{in }L^2(Q), \label{thetaeps->} \\
  &\thetahat_\eps\stackrel{*}{\rightharpoonup}\thetahat
  & \qquad & \textrm{in }L^\infty(0,T;V)\cap H^1(0,T;H),\\
  &\chieps\stackrel{*}\rightharpoonup\chi & \qquad& \textrm{in }L^\infty(Q). \label{chieps->} 
\end{alignat}
An integration in time of the energy balance equation \eqref{pb-eps eq.energ.} yields
\begin{equation}
  \thetaeps + \chieps + \A \thetahat_\eps = \theta_{0\eps} + \chi_{0\eps} + \fhat_\eps,
  \qquad \text{in $L^2(0,T;V')$}
\end{equation}
therefore taking the limit as $\eps \to 0$ along the subsequence established above we get
\begin{equation}
  \theta + \chi + \A \thetahat = \theta_{0} + \chi_{0} + \fhat
  \qquad \text{in $L^2(0,T;V')$}
\end{equation}
which turns out to be equivalent to \eqref{P pb 2}.
From the Lipschitz continuity \eqref{phi lip} of $\psi$ we have that
\begin{align}
  & \int_Q |\eps\chieps'(t,x) - \psi(\theta(t,x) + u(t,x),\chi(t,x))||v(t,x)| \de x \de t \notag \\
  & = \int_Q |\psi(\thetaeps(t,x) + u_\eps(t,x),\chieps(t,x)) - \psi(\theta(t,x) + u(t,x),\chi(t,x))||v(t,x)| \de x \de t \notag \\
  & \le 
    L\int_Q (|\thetaeps(t,x)-\theta(t,x)| + |u_\eps(t,x)-u(t,x)| + |\chi(t,x)-\chieps(t,x)|)|v(t,x)| \de x \de t
\end{align}
for every $v \in L^2(Q)$, therefore if $\xi \in L^2(Q)$ is defined by 
\begin{equation}
  \xi(t,x) := \psi(\theta(t,x) + u(t,x),\chi(t,x)), \qquad (t,x) \in Q, \label{xi}.  
\end{equation}
we have that
\begin{equation}
  \eps\chieps' \ {\rightharpoonup}\ \xi \qquad \text{in $L^2(Q)$} \label{epschi'->xi}
\end{equation}
(observe that $\thetaeps \to \theta$ and $\chieps \to \chi$ a.e. in $Q$ by virtue of \cite[(3.9)-(3.10)]{Vis01}). 
On the other hand from \eqref{chieps->} we have that $\chieps' \to \chi'$ in $Q$ in the sense of distributions, therefore 
$\eps\chieps' \to 0$ in $Q$ in the sense of distributions and we infer that
\begin{equation}
\eps\chieps' \ \rightharpoonup \ 0   
       \quad\textrm{in }L^2(Q). \label{epschi->0}
\end{equation}
Thus from \eqref{xi}, \eqref{epschi'->xi} and \eqref{epschi->0} we infer that
\begin{equation}
  \psi(\theta(t,x)+u(t,x),\chi(t,x))  = 0 \qquad \text{for a.e. $(t,x) \in Q$},
\end{equation}
so that by \eqref{psi(r,s)=0} we get that
\begin{equation}
  \chi(t,x) \in \alpha(\theta(t,x) + u(t,x)) \qquad \text{for a.e. $(t,x) \in Q$} 
\end{equation}
and also \eqref{P pb 3} is proved.
It remains to prove uniqueness, which also allows us to deduce that the whole sequences 
$(\theta_\eps)$ and $(\chi_\eps)$ converge. Let 
$(\theta_i,\chi_i),\ i=1,2,$ be two solutions, and set 
\begin{equation}
  \thetadiff:=\theta_1-\theta_2, \qquad \chidiff:=\chi_1-\chi_2.
\end{equation} 
Taking the difference of the equations \eqref{P pb 2}  written for 
$(\theta_1,\chi_1)$ and $(\theta_2,\chi_2),$ we find
\begin{eqnarray}
&\thetadiffhat\in L^{\infty}(0,T;V)\cap H^1(0,T;H),\quad  \\  
 & \chidiff\in L^\infty(0,T;H),\\
&\thetadiff+\chidiff+A\thetadiffhat=0\quad 
  \textrm{in}\ V',\ \ {\rm in}\ \opint{0,T}.\label{int-eq in H}
\end{eqnarray}
By a comparison in the last equation, we see that 
$A\thetadiffhat\in L^2(0,T;H),$ therefore 
multiplying \eqref{int-eq in H} by $\thetadiff$ and integrating over 
$\Omega\times(0,t)$,
we get 
\begin{equation}
\normaq{\thetadiff}{L^2(0,t;H)}+
\intdoppio\chidiff(s,x)\thetadiff(s,x) \de x \de s +  
\frac{1}{2}\intom|\nabla\thetadiffhat(t,x)| \de x=0.  \label{last eq}
\end{equation}
Therefore, since $\chidiff\thetadiff\geq0$ a.e. in $Q$ by the monotonicity of $\alpha$ and 
\eqref{P pb 3}, from \eqref{last eq} we infer that $\thetadiff=0$ a.e. in $Q$ and, by a comparison in 
\eqref{int-eq in H}, that $\chidiff=0$ a.e. in $Q$, and the uniqueness of Problem {\bf (P)} is proved.
This uniqueness property, together with the fact that the formulation \eqref{wStef th}--\eqref{wStef pb 3} is stronger than the formulation of Problem {\bf (P)}, let us infer that $(\theta,\chi)$ is indeed the solution of 
\eqref{wStef th}--\eqref{wStef pb 3}.
\end{proof}


\vspace{1ex}

\section*{Acknowledgment}

I would like to express my gratitude to Pierluigi Colli for introducing me, more than twenty years ago, to research in mathematics and in PDE's by  proposing and helping me with my first work \cite{Rec02}.

I am also grateful to an anonymous referee who read very carefully the manuscript and pointed out a number of inaccuracies.

\vspace{1ex}


\end{document}